\documentclass{article}
\usepackage{amsmath, amssymb ,amsthm, amsfonts, amsgen}
\usepackage{graphicx}
\usepackage[dvips]{color}

\newtheorem{Theorem}{Theorem}
\newtheorem{Lemma}[Theorem]{Lemma}
\newtheorem{Proposition}[Theorem]{Proposition}

\newtheorem{Remark}[Theorem]{Remark}
\theoremstyle{definition}
\newtheorem{Definition}[Theorem]{Definition}

\newcommand{\Rb}{{\mathbb{R}}}

\newcommand{\C}{{\mathcal{C}}}

\newcommand{\LL}{{\mathcal{L}}}
\newcommand{\HH}{{\mathcal{H}}}
\newcommand{\M}{{\mathcal{M}}}

\makeatletter
\def\rightharpoonupfill@{\arrowfill@\relbar\relbar\rightharpoonup}
\newcommand{\xrightharpoonup}[2][]{\ext@arrow
0359\rightharpoonupfill@{#1}{#2}} \makeatother

\def\weakstar{\buildrel\ast\over\rightharpoonup}
\def\e{{\varepsilon}}
\def\O{{\Omega}}

\def\weak{\rightharpoonup}

\def\M{{\it M}}

\def\M{{\cal M}}

\newcommand{{\rr}}{{\mathbb R}}

\begin{document}

%\keywords{Lower semicontinuity, convexity-quasiconvexity.}
%\mathclass{49J45; 74F99.}

%\abbrevauthors{A. M. Ribeiro and E. Zappale}
%\abbrevtitle{Weak $\ast$ lower semicontinuity in $W^{1,1}\times L^p$}

\title{Lower semicontinuous envelopes in $W^{1,1}\times L^p$}

%\authors{Ana Margarida Ribeiro^\footnote{Centro de  Matem\'{a}tica e %Aplica\c{c}\~{o}es (CMA), FCT, UNL \\ Departamento de Matem\'{a}tica, %FCT, UNL\\
%Quinta da Torre, 2829-516 Caparica, Portugal
%\\ E-mail: amfr@fct.unl.pt},

\author{ \textsc{ Ana Margarida Ribeiro} \thanks{Centro de  Matem\'{a}tica e Aplica\c{c}\~{o}es (CMA), FCT, UNL, Departamento de Matem\'{a}tica, FCT, UNL,
Quinta da Torre, 2829-516 Caparica, Portugal, E-mail: amfr@fct.unl.pt}, \textsc{ Elvira Zappale}\thanks{Universita'
degli Studi di Salerno, Via Ponte Don Melillo, 84084 Fisciano (SA) Italy.
E-mail:ezappale@unisa.it}}

% Elvira Zappale, \footnote{D. I. E. I. I., Universit\`a degli Studi di %Salerno\\ via Ponte Don Melillo, 84084, Fisciano (SA), Italy\\
%E-mail: ezappale@unisa.it}}

\maketitle

\begin{abstract}
\noindent It is studied the lower semicontinuity of functionals of the type $\displaystyle{\int_\Omega f(x,u,v,\nabla u) dx }$ with respect to the $(W^{1,1}\times L^p)$-weak $\ast$ topology. Moreover in absence of lower semicontinuity, it is also provided an integral representation in $W^{1,1}\times L^p$ for the lower semicontinuous envelope.


\end{abstract}

\section{Introduction}
In this paper we consider energies depending on two vector fields with different behaviours: $u \in W^{1,1}(\Omega; \mathbb R^n)$, $v \in L^p(\Omega;\mathbb R^m)$, $\Omega$ being a bounded open set of $\mathbb R^N$.
Let $1<p\leq +\infty$, for every $(u,v)\in W^{1,1}(\Omega;\mathbb R^n)\times L^p(\Omega;\mathbb R^m)$ define the functional
\begin{equation}\label{J}
\displaystyle{J(u,v):= \int_\Omega f(x, u(x), v(x), \nabla u(x))dx}
\end{equation}
where $f:\Omega\times \mathbb{R}^n \times \mathbb{R}^m\times \mathbb{R}^{n\times N}\rightarrow [0,+\infty)$ is a continuous function with linear growth in the last variable and $p$-growth in the third variable (cf. $(H1_p)$ and $(H1_\infty)$ below).

The energies \eqref{J}, which generalize those considered by \cite{Fonseca-Kinderlehrer-Pedregal_1}, \cite{Fonseca-Kinderlehrer-Pedregal_2}  and \cite{CRZ2}, have been introduced to deal with equilibria for systems depending on elastic strain and chemical composition. In this context a multiphase alloy is represented by the set $\Omega$, the deformation gradient is given by $\nabla u$, and $v$ (when $m=1$) denotes the chemical composition of the system.
We also recall that our result may find applications also in the framework of Elasticity, when dealing with Cosserat's theory, see \cite{LDR}.
In \cite{Fonseca-Kinderlehrer-Pedregal_1}, the density $f \equiv f(v, \nabla u)$ is a convex-quasiconvex function, while in our model we also take into account heterogeneities and the deformation, without imposing any convexity restriction.

We are interested in studying the lower semicontinuity and relaxation of \eqref{J} with respect to the $L^1$-strong $\times L^p$-weak convergence. Clearly, bounded sequences $\{u_n\}\subset $ $ W^{1,1}(\O;\mathbb{R}^n)$ may converge in $L^1$, up to a subsequence,  to a $BV$ function. In this paper we restrict our analysis to limits $u$ which are in $W^{1,1}(\O;\mathbb{R}^n)$. Thus, our results can be considered as a step towards the study of relaxation in $BV(\Omega;\mathbb R^n) \times L^p(\Omega;\mathbb R^m)$ of functionals \eqref{J}.

We will consider separately the cases $1<p<\infty$ and $p=\infty$. To this end we introduce for $1<p<+\infty$ the functional
\begin{equation}\label{Jpbar}
\begin{array}{c}\overline{J}_p(u,v):=\inf\left\{\liminf J(u_n,v_n):\ u_n\in W^{1,1}(\O;\mathbb{R}^n),\ v_n\in L^p(\O;\mathbb{R}^m),\right.\vspace{0.2cm}\\ \hspace{5cm} \left.u_n \to u \hbox{ in } L^1, v_n \weak v  \hbox{ in } L^p\right\},\end{array}
\end{equation}
for any pair $(u,v)\in W^{1,1}(\O;\mathbb{R}^n)\times L^p(\O;\mathbb{R}^m),$
and for $p=\infty$ the functional
\begin{equation}\label{Jinftybar}
\begin{array}{c}\overline{J}_\infty(u,v):=\inf\left\{\liminf J(u_n,v_n):\ u_n\in W^{1,1}(\O;\mathbb{R}^n),\ v_n\in L^\infty(\O;\mathbb{R}^m),\right.\vspace{0.2cm}\\ \hspace{5cm} \left.u_n \to u \hbox{ in } L^1, v_n \weakstar v  \hbox{ in } L^\infty\right\},\end{array}
\end{equation}
for any pair $(u,v)\in W^{1,1}(\O;\mathbb{R}^n)\times L^\infty(\O;\mathbb{R}^m)$.

\noindent For any $p \in (1,+\infty]$ we will achieve the following integral representation (see Theorems \ref{maintheoreminftyrelax} and \ref{maintheoremprelax}):
$$
\displaystyle{\overline{J}_p(u,v)=\int_\Omega CQf(x,u(x), v(x),\nabla u(x))dx,}
$$
where $CQf$ represents the convex-quasiconvexification of $f$ defined in \eqref{CQ}.
\section{Notations and General Facts}

In this section we introduce the sets of assumptions we will make to obtain our results. We prove some properties related to convex-quasiconvex functions and we recall several facts that will be useful through the paper.

\subsection{Assumptions} Let $1<p< +\infty$, to obtain a characterization of the relaxed functional $\overline{J}_p$ in \eqref{Jpbar}, we will make several hypotheses on the continuous function  $f:\Omega\times \mathbb{R}^n \times \mathbb{R}^m\times \mathbb{R}^{n\times N}\rightarrow [0,+\infty)$. They are inspired by the set of assumptions in \cite{Fonseca-Muller-quasiconvex} for the case with no dependence on $v$.\medskip

\noindent $({\mathrm H}1_p)$ There exists  a constant $C$ such that
$$\frac{1}{C}(|v|^p + |\xi|)-C\le f(x,u,v,\xi)\le C(1+|v|^p +|\xi|),$$
for every  $(x,u,\xi)\in\Omega\times\mathbb{R}^n\times\mathbb R^{m}\times \mathbb{R}^{n\times N}$
\smallskip

\noindent $({\mathrm H}2_p)$ For every compact set $K$ of $\Omega \times \mathbb R^n$ there exists a continuous function $\omega_{K}:\mathbb{R}\rightarrow[0,+\infty)$ with $\omega_{K}(0)=0$ such that
\begin{equation}\nonumber
|f(x,u,v,\xi)-f(x',u',v,\xi)|\leq \omega_K(|x-x'|+|u-u'|)(1+|v|^p +|\xi|)
\end{equation}
 for every $(x,u,v,\xi)$ and $(x', u',v,\xi)$ in $K \times \mathbb R^m \times \mathbb R^{n \times N}$.\medskip

\noindent Moreover, given $x_0\in\Omega$, and $\varepsilon>0$ there exists  $\delta>0$ such that if $|x-x_0|\le\delta$ then
\begin{equation}\nonumber
\forall\ (u,v,\xi)\in\mathbb{R}^n\times\mathbb R^m \times\mathbb{R}^{n\times N}\quad f(x,u,v,\xi)-f(x_0,u,v,\xi)\ge -\varepsilon (1+|v|^p+|\xi|).
\end{equation}

\medskip
In order to characterize the functional ${\overline J_\infty}$ defined in \eqref{Jinftybar} we will replace assumptions $(\mathrm{H}1_p)$ and $(\mathrm{H}2_p)$ by the following ones.

\noindent$(\mathrm{H}1_\infty)$ Given $M>0$, there exist $C_M>0$ and a bounded continuous function $G_M: \Omega \times \mathbb R^n \to [0, +\infty)$ such that, if $|v|\le M$ then $$\forall\ (x,u,\xi)\in\Omega\times\mathbb{R}^n\times\mathbb{R}^{n\times N}\quad\frac{1}{C_M}G_M(x,u)|\xi|-C_M\le f(x,u,v,\xi)\le C_M G_M(x,u)(1+|\xi|).$$\smallskip

\noindent$(\mathrm{H}2_\infty)$ For every  $M>0$, and for every compact set $K$ of $\Omega \times \mathbb R^n$ there exists a continuous function $\omega_{M,K}:\mathbb{R}\rightarrow[0,+\infty)$ with $\omega_{M,K}(0)=0$ such that if $|v|\le M$ then
$$
|f(x,u,v,\xi)-f(x_0,u_0,v,\xi)|\leq \omega_{M,K}(|x-x_0|+|u-u_0|)(1+|\xi|)
$$
 for every $(x,u,\xi), (x_0, u_0,\xi) \in K \times \mathbb R^{n \times N}$.

\noindent Moreover, given $M>0$, $x_0\in\Omega$, and $\varepsilon>0$ there exists  $\delta>0$ such that if $|v|\le M$ and $|x-x_0|\le\delta$ then
$$\forall\ (u,\xi)\in\mathbb{R}^n\times\mathbb{R}^{n\times N}\quad f(x,u,v,\xi)-f(x_0,u,v,\xi)\ge -\varepsilon G_M(x,u) (1+|\xi|),$$
where the function $G_M$ is as in $(H1_\infty)$.

\subsection{Convex-Quasiconvex Functions}

We start recalling the notion of convex-qua\-si\-convex function,  presented  in \cite{Fonseca-Kinderlehrer-Pedregal_1} (see also \cite[Definition 4.1]{LDR}, \cite{Fonseca-Kinderlehrer-Pedregal_2} and \cite{Fonseca-Francfort-Leoni}).
This notion plays,  in the context of lower semicontinuity problems where the density depends on two fields $v, \nabla u$, the same role of the well known notion of quasiconvexity introduced by Morrey for the lower semicontinuity of functionals where the dependence is just on $\nabla u$.

\begin{Definition}
A Borel measurable function $h: \mathbb R^m \times \mathbb R^{n \times N} \to \mathbb R$ is said to be \emph{convex-quasiconvex} if there exists a bounded open set $D$ of $\mathbb R^N$ such that
\begin{equation}\label{cross-qcx}
h(v,\xi)\le\frac{1}{|D|}\int_D h(v+ \eta(x), \xi + \nabla \varphi(x))\,dx,
\end{equation}
for every $(v,\xi)\in \mathbb R^m \times \mathbb R^{n \times N}$, for every $\eta \in L^\infty(D;\mathbb R^m)$, with $\displaystyle{\int_D \eta(x)\,dx=0}$, and for every $\varphi \in W^{1, \infty}_0(D;\mathbb R^{n})$.
\end{Definition}

\begin{Remark}{\rm (i) It can be easily seen that, if $h$ is convex-quasiconvex, then, condition (\ref{cross-qcx}) is true for any bounded open set $D\subset\mathbb{R}^N$.

\noindent(ii) We recall that a convex-quasiconvex function is separately convex.

\noindent(iii) Through this paper we will work with functions $f$ defined in $\Omega \times \mathbb R^n \times \mathbb R^m \times \mathbb R^{n \times N}$ and when saying that $f$ is convex-quasiconvex we mean the previous definition with respect to the last two variables of $f$.}
\end{Remark}

The following result adapts to the context of $W^{1,1}\times L^p$, i.e. growth conditions expressed by $({\rm H}1_p)$, a well known result due to Marcellini (see Proposition 2.32 in \cite{Dacorogna} or Lemma 5.42 in \cite{AFP}).  Indeed the following proposition follows as a particular case of \cite[Proposition 2.11]{CRZ1} .
%for any $1<p <+\infty$ and $1\leq q < +\infty$. We state it here when %$q=1$ and $1\leq p< +\infty$.

\begin{Proposition}\label{locallylipschitzLP}
Let $f:\Omega \times \mathbb R^n \times\mathbb{R}^m \times \mathbb{R}^{n\times
N}\rightarrow\mathbb{R}$ be a separately convex function in each entry of the variables $(v, \xi)$, verifying the
growth condition
\begin{equation}\nonumber
|f(x,u,v, \xi)|\leq c(1+|\xi|+|v|^p),\ \forall\ (x,u,\xi,v)\in\Omega \times \mathbb R^m \times \mathbb{R}^{n\times
N}\times\mathbb{R}^{m} %\label{p1growth}
\end{equation}
for some $p>1$.

\noindent Then, denoting by $p^{\prime}$, the conjugate exponent of $p$, there exists a constant $\gamma>0$ such that
$$
\left\vert f\left( x,u, v,\xi\right)  -f\left(x,u, v', \xi'\right)
\right\vert \leq\gamma  \left\vert
\xi-\xi^{\prime}\right\vert +\gamma\left(  1+\left\vert v\right\vert
^{p-1}+\left\vert v'\right\vert ^{p-1}+\left\vert \xi'\right\vert ^{1/p^{\prime}}\right)  |v-v^{\prime}|
$$
 for every $\xi,\xi^{\prime}%
\in\mathbb{R}^{n\times N}$, $v,v^{\prime}\in\mathbb{R}^{m}$ and $(x,u)\in \Omega \times \mathbb R^n$.
\end{Proposition}

A similar result holds in the case $W^{1,1}\times L^\infty$ (i.e. growth conditions expressed by $({\rm H1}_\infty)$). 

\begin{Proposition}\label{locallyLipschitz}
Let $f:\Omega \times \mathbb R^n \times \mathbb R^m \times \mathbb R^{n \times N}\to \mathbb R$ be a separately convex function in each entry of the variables $(v,\xi)$, verifying assumption $({\rm H1}_\infty)$. Then, given $M >0$ there exists a constant $\beta(M,n, m,N)$ such that
\begin{equation}\label{MarcelliniLipschitz}
|f(x,u,v, \xi)-f(x,u,v',\xi')|\leq \beta  (1+ |\xi|+ |\xi'|)|v-v'|+ \beta |\xi-\xi'|,
\end{equation}
for  every $v, v' \in \mathbb R^m$, such that $|v|\leq M$ and $|v'|\leq M$, for every $\xi ,\xi' \in \mathbb R^{n \times N}$ and for every $(x,u) \in \Omega \times \mathbb R^n$.
\end{Proposition}

We introduce the notion of convex-quasiconvexification with respect to the last variables for a function $f:\Omega \times \mathbb R^n \times \mathbb R^m \times \mathbb{R}^{n \times N}\to [0, +\infty)$. This notion is crucial in order to deal with the subsequent relaxation processes.

If $h:\mathbb{R}^{m}\times \mathbb R^{n \times N}\rightarrow\mathbb{R}$ is any
given Borel measurable function bounded from below, it can be defined the
convex-quasiconvex envelope of $h$, that is the largest convex-quasiconvex
function below $h$:
\begin{equation}\label{CQ}
CQh(v,\xi):=\sup\{g(v,\xi):\ g\leq h,\ g\hbox{ convex-quasiconvex}\}.
\end{equation}
Moreover, by Theorem 4.16 in \cite{LDR}
\begin{equation}
\begin{array}{ll}
CQh(v,\xi)=\inf & \left\{  \displaystyle{\frac{1}{|D|}\int_{D}h(v+ \eta(x),\xi+\nabla\varphi(x))\,dx}:\right. \\
& \left. \displaystyle{\eta\in L^{\infty
}(D;\mathbb{R}^{m}),\int_{D}\eta(x)dx=0, \varphi\in W_{0}^{1,\infty}(D;\mathbb{R}^n),}\right\}  . \label{qcx-cx-rel}
\end{array}
\end{equation}

Consequently given a Carath\'eodory function $f:\Omega \times \mathbb R^n \times \mathbb R^m \times \mathbb R^{n \times N} \to \mathbb R$, by $CQf(x,u,v,\xi)$ we denote the convex-quasiconvexification of $f(x,u,v,\xi)$ with respect to the last two variables.

As for convex-quasiconvexity, condition (\ref{qcx-cx-rel}) can be stated for any bounded open set $D\subset\mathbb{R}^{N}$ and it can be also showed that if $f$ satisfies a growth condition of the type $({\rm H1}_p) $ then in (\ref{cross-qcx}) and
(\ref{qcx-cx-rel}) the spaces $L^{\infty}$ and $W_{0}^{1,\infty}$ can be
replaced by $L^p$ and $W_{0}^{1,1}$, respectively.

The following results will be exploited in the sequel. We omit the proofs since they are very similar to \cite[Proposition 2.2]{RZ}, in turn inspired by \cite{Dacorogna}.

\begin{Proposition}\label{measCQf_p}
Let $\Omega \subset \mathbb R^N$ be a bounded open set and
$
f: \Omega \times \mathbb R^n \times \mathbb R^m \times \mathbb R^{n \times N}$ $ \to \mathbb [0, +\infty)
$
be a continuous function satisfying $({\rm H1}_p)$ and $({\rm H2}_p)$. Let $CQf$ be the convex-quasiconvexification of $f$  in \eqref{qcx-cx-rel}. Then $CQf$ satisfies $({\rm H1}_p)$, $({\rm H2}_p)$ and it is a continuous function.

%and satisf Then the validity of $\bf {(H1)_p}$ guarantees that there %exists  a constant $C>0$ such that
%\begin{equation}\label{H1pCQf}
%\frac{1}{C}(|\xi|+ |v|^p)- C \leq CQf(x,u,v,\xi) \leq C(1+ |\xi|+|v|^p),\ %\forall\ (x,u,\xi) \in \Omega  \times \mathbb R^d \times \mathbb R^m %\times \mathbb R^{d \times N}.
%\end{equation}

%The validity of \eqref{H2_1p}  ensures that
%for every compact set $K \subset \Omega \times \mathbb R^d$, there %exists a continuous function $\omega'_K:\mathbb{R}\to [0, +\infty)$ %such that
%$\omega'_K(0)=0$ and
%\begin{equation}\label{H2_1PCQf}
%|CQf(x,u,v,\xi)- CQf(x',u',v,\xi)| \leq \omega'_K(|x-x'|+|u-u'|)(1 +|v|^p+ %|\xi|),\
%\end{equation}
%$ \hbox{ for every }\ (x,u),(x',u')\in K,\ \forall\ (v,\xi) \in \mathbb R^m %\times \mathbb R^{d \times N}.$

%Condition \eqref{H2_2p} entails that, for every $x_0 \in \Omega$ and %$\e >0$, there exists $\delta>0$ such that
%\begin{equation}\label{H2_2PCQf}
%|x-x_0|\leq \delta\ \Rightarrow\ \forall\ (u,v,\xi)\in \mathbb R^d \times %\mathbb R^m \times \mathbb R^{d \times N} \; \; CQf(x,u,v,\xi) - %CQf(x_0,u,v,\xi) \geq -\e(1+|v|^p +|\xi|).
%\end{equation}
%.

\end{Proposition}

Analogously it holds

\begin{Proposition}\label{measCQf_infty}
Let $\Omega \subset \mathbb R^N$ be a bounded open set, let $\alpha:[0, +\infty) \to [0,+\infty)$ be a convex and increasing function, such that $\alpha(0)=0$ and let
$f: \Omega \times \mathbb R^n \times \mathbb R^m \times \mathbb R^{n \times N} \to \mathbb [0, +\infty)$
be a continuous function satisfying the following conditions.

\noindent For a.e. $(x,u)\in \Omega \times \mathbb R^n$ and for every $(v,\xi)\in \mathbb R^m \times \mathbb R^{n \times N}$ it results
\begin{equation}\label{H1_1inftyCQf}
\frac{1}{C}(\alpha(|v|)+ |\xi|)-C \leq f(x,u,v,\xi) \leq C(1+ \alpha(|v|)+ |\xi|).
\end{equation}

\noindent For every compact set $K \subset \Omega \times \mathbb R^n$ there exists a continuous function $\omega'_{K}:\mathbb R \to [0, +\infty)$ such that
$\omega'_{K}(0)=0$ and
\begin{equation}\label{H2_1inftyCQf}
|f(x,u,v,\xi)- f(x',u',v,\xi)| \leq \omega'_{K}(|x-x'|+|u-u'|)(1 + \alpha(|v|)+ |\xi|),\
\end{equation}
$ \forall (x,u),(x',u')\in K,\ \forall\ (v,\xi) \in \mathbb R^m \times \mathbb R^{n \times N}.$

\noindent For every $x_0 \in \Omega$ and $\e >0$, there exists $\delta>0$ such that
\begin{equation}\label{H2_2inftyCQf}
 |x-x_0|\leq \delta\ \Rightarrow\  f(x,u,v,\xi) - f(x_0,u,v,\xi) \geq -\e(1 +\alpha(|v|) +|\xi|),
\end{equation}
$\forall\ (u,\xi)\in \mathbb R^n \times \mathbb R^m \times \mathbb R^{n \times N}.$

Let $CQf$ be the convex- quasiconvexification of $f$  (see \eqref{qcx-cx-rel}). Then $CQf$ satisfies analogous conditions to \eqref{H1_1inftyCQf}, \eqref{H2_1inftyCQf} and \eqref{H2_2inftyCQf}. Moreover CQf is a continuous function.
\end{Proposition}

\begin{Remark}
We observe that, if from one hand \eqref{H1_1inftyCQf}, \eqref{H2_1inftyCQf}, \eqref{H2_2inftyCQf} generalize $(H1_p)$ and $(H2_p)$, from the other hand they can be regarded also as a stronger version of $(H1_\infty)$ and $(H2_\infty)$.
\end{Remark}

In order to provide an integral representation for $\overline{J}_p$ in \eqref{Jpbar} and $\overline{J}_\infty$ in \eqref{Jinftybar} on $W^{1,1}\times L^p$ and $W^{1,1}\times L^\infty$ respectively, we prove some preliminary results.

\noindent For every $p \in (1,+\infty]$ we introduce the following functionals $J_{CQf}:L^1(\Omega; \mathbb R^n)\times L^p(\Omega;\mathbb R^m)\rightarrow\mathbb{R}\cup\{+\infty\}$ defined as
$$
J_{CQf}(u,v):=\left\{\begin{array}{ll}
\displaystyle{\int_\O CQf(x,u(x),v(x),\nabla u(x))\,dx }&\hbox{ if }(u,v) \in W^{1,1}\times L^p,\vspace{0.2cm} \\
+ \infty &\hbox{otherwise},
\end{array}
\right.
$$ and its relaxed one $$
\begin{array}{ll}\displaystyle{\overline{J_{CQf}}(u,v):=\inf\left\{\liminf_n J_{CQf}(u_n,v_n): \right.}&\displaystyle{(u_n, v_n)\in W^{1,1}(\Omega;\mathbb{R}^n)\times L^p(\Omega;\mathbb R^m),}\\
\\
&\displaystyle{\left. \ u_n \to u \hbox{ in }L^1, v_n \weakstar v \hbox{ in }L^p\right\}}.
\end{array}$$

 %We are now in position to establish the mentioned lemma and we notice %that we make no assumptions on the quasiconvexified function $QW$.

\begin{Lemma}\label{FMnonquasiconvex}
Let $f:\Omega\times\mathbb{R}^n\times \mathbb R^m \times \mathbb{R}^{n\times N}\rightarrow [0,+ \infty)$ be a continuous function. Let $p \in (1,+\infty]$ and consider the functionals $J$ and $J_{CQf}$ and their corresponding relaxed functionals $\overline{J}_p$  and $\overline{J_{CQf}}$. If $f$ satisfies conditions $(H1_p)-(H2_p)$ (if $p \in (1,+\infty))$, and both $f$ and $CQf$ satisfy $(H1_\infty)-(H2_\infty)$ (if $p=+\infty$), then
$$\displaystyle{\overline{J}_p(u,v)=\overline{J_{CQf}}(u,v)}$$ for every $(u,v) \in BV(\Omega,\mathbb{R}^n)\times L^p(\Omega;\mathbb R^m)$.
\end{Lemma}

\begin{Remark}\label{growthremark}
We emphasize that in the above lemma,  by virtue of Proposition \ref{measCQf_p},  if $p \in (1,+\infty)$ it is enough to assume growth and continuity hypotheses just on $f$ (and not on $CQf$). If $p=+\infty$, by virtue of Proposition \ref{measCQf_infty}, we can also only make assumptions of $f$, replacing conditions $(H1_\infty)$  and $(H2_\infty)$ by \eqref{H1_1inftyCQf} $- $\eqref{H2_2inftyCQf}.
\end{Remark}

\begin{proof}[Proof]
The argument is close to the proof of \cite[Lemma 3.1]{RZ}. First we observe that, since $CQf \leq f$, it results $\overline{J_{CQf}}\leq\overline{J}_p$. Next we prove the opposite inequality in the nontrivial case that $\overline{J_{CQf}}(u,v)<+\infty$. For fixed $\delta>0$, we can consider $(u_n,v_n)\in W^{1,1}(\Omega;\mathbb{R}^n)\times L^p(\Omega;\mathbb R^m)$ with $u_n \to u$ strongly in $L^1(\Omega;\mathbb R^n)$, $v_n \weakstar v$ in $L^p(\Omega;\mathbb R^m)$ and such that
$$
\displaystyle{\overline{J_{CQf}}(u,v)\geq \lim_n \int_\Omega CQf(x,u_n(x),v_n(x), \nabla u_n(x))\,dx - \delta.}
$$
Applying the results in \cite{CRZ1} and \cite{CRZ2},
for each $n$ there exists a sequence $\{(u_{n,k},v_{n,k}\}$ converging to $(u_n,v_n)$ weakly in $W^{1,1}(\Omega;\mathbb R^n)\times L^p(\Omega;\mathbb R^m)$ such that
$$
\displaystyle{\int_\O CQf(x, u_n(x), v_n(x), \nabla u_n(x))\,dx = \lim_k \int_\O f(x, u_{n,k}(x), v_{n,k}(x), \nabla u_{n,k}(x))\,dx.}
$$
Consequently
\begin{equation}\label{2.17BFMbend2}
\displaystyle{\overline{J_{CQf}}(u,v)\geq \lim_n \lim_k \int_\O f(x, u_{n,k}(x), v_{n,k}(x), \nabla u_{n,k}(x))\,dx- \delta,}
\end{equation}
$$
\displaystyle{\lim_n\lim_k \| u_{n,k}- u\|_{L^1}=0.}
$$
and
$$
\displaystyle {v_{n,k} \weakstar v \hbox{ in }L^p \hbox{ as }k \to +\infty \hbox{ and }n \to +\infty}.
$$
Via a diagonal argument (remind that weak $L^p$ and weak $\ast L^\infty$- topologies are metrizable on bounded sets), there exists a sequence $\{(u_{n, k_n},v_{n,k_n} )\}$ satisfying $u_{n,k_n} \to u$ in $L^1(\Omega;\mathbb R^n)$, $v_{n, k_n} \weakstar v$ in $L^p(\Omega;\mathbb R^m)$ and realizing the double limit in the right hand side of (\ref{2.17BFMbend2}).
Thus, it results
$$
\displaystyle{\overline{J_{CQf}}(u,v) \geq \lim_n \int_\O f(x, u_{n ,k_n}(x), v_{n,k_n}(x), \nabla u_{n, k_n}(x))\,dx -\delta \geq \overline{J}_p(u,v)- \delta.}
$$
Letting $\delta$ go to $0$ the conclusion follows.
\end{proof}

\subsection{Some Results on Measure Theory}

Let $\O$ be a generic open subset of $\mathbb R^N$, we denote by
$\M(\O)$ the space of all signed Radon measures in $\O$ with bounded
total variation. By the Riesz Representation Theorem, $\M(\O)$ can
be identified to the dual of the separable space $\C_0(\O)$ of
continuous functions on $\O$ vanishing on the boundary $\partial
\O$. The $N$-dimensional Lebesgue measure in $\Rb^N$ is designated
as $\LL^N$ while $\HH^{N-1}$ denotes the $(N-1)$-dimensional
Hausdorff measure. If $\mu \in \M(\O)$ and $\lambda \in \M(\O)$ is a
nonnegative Radon measure, we denote by $\frac{d\mu}{d\lambda}$ the
Radon-Nikod\'ym derivative of $\mu$ with respect to $\lambda$. By a
generalization of the Besicovich Differentiation Theorem (see
\cite[Proposition 2.2]{Ambrosio-Dal Maso}), it can be proved that  there exists a
Borel set $E \subset \O$ such that $\lambda(E)=0$ and
$$
\frac{d\mu}{d\lambda}(x)=\lim_{\rho \to 0^+} \frac{\mu(x+\rho \, C))}{\lambda(x+\rho \, C))}
$$
for all $x \in {\rm Supp }\, \mu \setminus E$ and any open convex
set $C$ containing the origin. (Recall that the set $E$ is independent of $C$.)

We also recall the following generalization of Lebesgue-Besicovitch Differentiation Theorem, as stated in \cite[Theorem 2.8]{Fonseca-Muller-relaxation}.

\begin{Theorem}\label{FMrthm2.8}
If $\mu$ is a nonnegative Radon measure and if $f \in L^1_{loc}(\mathbb R^d;\mu)$ then
$$
\displaystyle{\lim_{\e \to 0^+}\frac{1}{\mu(x+\e C)}\int_{x+ \e C}|f(y)-f(x)|d \mu(y)=0,}
$$
for $\mu$-a.e. $x \in \mathbb R^d$ and for every bounded, convex, open set $C$ containing the origin.
\end{Theorem}

In particular, if $v \in L^\infty(\O;\mathbb{R}^m)$, then, for $\mathcal{L}^N$-a.e. $x\in\O$
\begin{equation}\label{continuity-point}
\lim_{\varepsilon\rightarrow 0}\frac{1}{|B_\varepsilon(x)|}\int_{B_\varepsilon(x)} |v(y)-v(x)|\,dy=0.
\end{equation}
In the sequel we exploit the Calder\'on-Zygmund theorem for $u \in BV$, cf. \cite[Theorem 3.83, page 176]{AFP}
\begin{equation}\label{differentiability-point}
\lim_{\varepsilon\rightarrow 0}\frac{1}{\varepsilon|B_\varepsilon(x)|}\int_{B_\varepsilon(x)} |u(y)-u(x)-\nabla u(x)(y-x)|\,dy=0\quad \mathcal{L}^N-a.e.\ x\in\Omega.
\end{equation}

%For the readers' convenience we recall the following result, whose proof %can be found in \cite[Lemma 4.5]{Ambrosio-Mortola-Tortorelli}.

%\begin{Lemma}\label{Lemma2.5FMr}
%Let $u \in BV(\O;\mathbb R^n)$ and let $\varrho \in C^\infty_0(\mathbb %R^N)$ be a nonnegative function such that
%$$
%\int_{\mathbb R^N}\varrho(x)dx=1, \;\; {\rm supp} \varrho = %\overline{B}(0,1), \;\; \varrho(x)= \varrho(-x) \hbox{ for every }x \in %\mathbb R^N.
%$$
%Let $\varrho_k(x):= k^N \varrho(kx)$ and
%$$
%u_k(x):= (u \ast \varrho_k)(x)=\int_\Omega u(y) \varrho_k(x-y)dy.
%$$
%Then
%\begin{itemize}
%\item[(i)] $\displaystyle\int_{B(x_0,\e)}h(x)|\nabla u_k(x)|dx \leq %\int_{B\left(x_0, \e +\frac{1}{k}\right)}(h \ast \varrho_k)(x)|Du (x)|$, %whenever ${\rm dist}(x_0,\partial \Omega) > \e + \frac{1}{k}$ and $h$ %is a nonnegative Borel function.
%\item[(ii)] $\displaystyle{\lim_{k \to + \infty}\int_{B(x_0,\e)}\theta(\nabla %u_k(x))dx = \int_{B(x_0,\e)}\theta (D u(x))}$ for every function $\theta$ %positively $1$-homogeneous  and for every $\e \in (0, {\rm dist}(x_0, %\partial \Omega))$ such that $|Du|(\partial B(x_0,\e))=0$.
%\item[(iii)] If, in addition, $u \in L^\infty(\Omega;\mathbb R^n)$, then for %every $x_0 \in \Omega \setminus S(u)$,
%$u_k(x_0)\to u(x_0)$, $(|u_k- u|\ast \varrho_k)(x_0) \to 0$, as $k \to + %\infty$.
%\end{itemize}
%\end{Lemma}

\section{Lower semicontinuity in $W^{1,1}\times L^p$}\label{11p}
This section is devoted to provide a lower bound for the integral representation of
${\overline J}_p$ in \eqref{Jpbar} under assumptions $({\mathrm H}1_p)$ and $({\mathrm H}2_p)$, as stated in Theorem \ref{maintheoremprelax}. Clearly this is equivalent to prove the lower semicontinuity with respect to the $L^1$-strong $\times$ $L^p$-weak topology of $\displaystyle{\int_\Omega CQf(x,u(x), v(x), \nabla u(x))dx}$, when $(u,v)\in W^{1,1}\times L^p$.

%Let $\Omega $ be a bounded open set of $\mathbb R^N$, and let $f:%\Omega\times \mathbb{R}^n \times \mathbb{R}^m\times %\mathbb{R}^{n\times N}\longrightarrow [0,+\infty)$ be a continuous %convex-quasiconvex function.
%We aim to characterize ${\overline J}_p$.

Indeed we prove the following result

\begin{Theorem}\label{maintheoremp}
Let $\Omega $ be a bounded open set of $\mathbb R^N$, and let $f:\Omega\times \mathbb{R}^n \times \mathbb{R}^m\times \mathbb{R}^{n\times N}\rightarrow [0,+\infty)$ be a continuous function. Assuming that $f$ satisfies hypotheses $(\mathrm{H}1_p)$ and $(\mathrm{H}2_p)$, and it is convex-quasiconvex, it results that
$ \displaystyle{\int_\O f(x,u(x),v(x),\nabla u(x))\,dx}$ is lower semicontinuous
in $ W^{1,1}(\Omega;\mathbb R^n)\times L^p(\Omega;\mathbb R^m)$, with respect to the ($L^1$-strong $\times$ $L^p$-weak)- convergence.
\end{Theorem}
 \begin{proof}[Proof] The proof is mostly a combination of the theorems in \cite{Fonseca-Muller-relaxation} and \cite{Fonseca-Kinderlehrer-Pedregal_1}, which used already some ideas from \cite{Fonseca-Muller-quasiconvex}.
%and \cite{Ambrosio-Mortola-Tortorelli}.
For convenience of the reader we present here some details, however we may refer to some separate results in the papers mentioned above.

Let $$G(u,v)  =  \displaystyle{\int_\O f(x,u(x),v(x),\nabla u(x))\,dx.}
$$

It's enough to prove that for every $(u,v)\in W^{1,1}(\Omega;\mathbb R^n)\times L^p(\Omega;\mathbb R^m)$, $G(u,v)\le\liminf J(u_n,v_n)$ for any $u_n \to u$ in $L^1$ with $u_n\in W^{1,1}(\Omega;\mathbb{R}^n)$ and $v_n \rightharpoonup v$ in $L^p$.

Using the same arguments as in \cite[Proof of Theorem II.4]{Acerbi-Fusco} (see also \cite[Proposition 2.4]{Fonseca-Muller-quasiconvex}) and the density of smooth functions in $L^p$, we can reduce to the case where $u_n\in C_0^\infty(\mathbb{R}^N;\mathbb{R}^n)$ and $v_n \in C^\infty_0(\mathbb{R}^N;\mathbb R^m)$.

Moreover, we can also suppose
$$\displaystyle{\liminf_{n \to \infty} J(u_n,v_n)=\lim_{n \to \infty} J (u_n,v_n)<+\infty.}$$ Then $J(u_n,v_n)$ is bounded and so, up to a subsequence, $\mu_n:=f(x,u_n,v_n,\nabla u_n)dx\weakstar \mu$ in the sense of measures for some positive measure $\mu$.

%By the decomposition $|Du|=|\nabla u|\mathcal{L}^N +|u^+-u^-| %\mathcal{H}^{N-1}\lfloor{J_u}+|D^cu|$ and
By the Radon-Nikodym theorem, $\mu=g\mathcal{L}^N+\mu_s$ for some $g\in L^1(\O)$, with $\mu_s$ singular with respect to ${\mathcal L}^N$.
%$h$ such that $h|u^+-u^-|\in L^1_{\mathcal{H}^{N-1}}(J_u)$ and $l\in %L^1_{|D^cu|}(\O)$.
It will be enough to prove the following inequality:

\begin{equation}\label{absolutely-continuousp}g(x)\ge f(x,u(x),v(x),\nabla u(x)),\ \mathcal{L}^N-a.e.\ x\in\O.\end{equation}

Indeed, once proved \eqref{absolutely-continuousp}, since $\mu_n \weakstar\mu$, by the lower semicontinuity of $\mu$, and since $\mu_s$ is nonnegative

%  and considering an increasing sequence of cut-off functions %$\varphi_k\in C_0(\O)$ such that $0\le\varphi_k\le 1$ and %$\sup_k\varphi_k=1$ in $\O$, then
$$\begin{array}{rcl}
\displaystyle{\lim_{n\rightarrow +\infty} J(u_n,v_n)} & = & \displaystyle{\lim_{n\rightarrow +\infty} \int_\O \,f(x,u_n(x),v_n(x),\nabla u_n(x))\,dx}\vspace{0.2cm}\\
& \geq & \displaystyle{\int_\O d\mu(x) = \int_\O g(x)\,dx
+\int_\O d\mu_s(x)}\vspace{0.2cm}\\
& \ge & \displaystyle{\int_\O f(x,u(x),v(x),\nabla u(x))\,dx.}
%+\int_{J_u}\varphi_k(x)\,K(x,u^-(x),u^+(x),0,%\nu_u(x))\,d\mathcal{H}^{N-1}(x)+}\vspace{0.2cm}\\
%& & +\displaystyle{\int_\O \varphi_k(x)\,f^\infty\left(x,u(x),0,%\frac{dD^cu}{d|D^cu|}(x)\right)d|D^cu|(x)}
\end{array}$$

In order to prove \eqref{absolutely-continuousp}, we follow the proofs of Theorem 2.1 in \cite{Fonseca-Muller-quasiconvex} and condition (2.3) in \cite{Fonseca-Kinderlehrer-Pedregal_1}. We start freezing the terms $x$ and $u$. This will be achieved through Steps 1 to 5.

By the Besicovitch derivation theorem
\begin{equation}\label{characterization-densityp}
g(x)=\lim_{\varepsilon\rightarrow 0}\frac{\mu(B_\varepsilon(x))}{|B_\varepsilon(x)|}\in\mathbb{R}\quad \mathcal{L}^N-a.e.\ x\in\O.
\end{equation}

%moreover we recall that (\ref{differentiability-point}) and (\ref{continuity-point}) hold.

Let $x_0$ be any element of $\Omega$ satisfying (\ref{characterization-densityp}), (\ref{differentiability-point}) and (\ref{continuity-point}) (notice that such an $x_0$ can be taken in $\Omega$ up to a set of Lebesgue-measure zero) and let's prove that $g(x_0)\ge f(x_0,u(x_0),v(x_0),\nabla u(x_0))$.
First remark that, as noticed before, since $v_n\rightharpoonup v$ in $L^p$, we have $||v_n||_{L^p},||v||_{L^p}\le C$.
%Let then $G_M$ be as in $(H1_\infty)$ and notice that if %$G_M(x_0,u(x_0))=0$ then $f(x_0,u(x_0),v(x_0),\nabla u(x_0))=0$ and %(\ref{absolutely-continuous}) is obviously true. We can then restrict %ourselves to the case where $G_M(x_0,u(x_0))>0$. This will be important %below in Step 4.
%\medskip

\noindent {\bf Step 1. Localization.} This part can be reproduced in the same way as in \cite{Fonseca-Muller-quasiconvex}: pages 1085-1086. We present some details for the reader's convenienece.  We start providing a first estimate for $g$. Observe that we can choose a sequence $\varepsilon\rightarrow0^{+}$ such that $\mu\left(  \partial
B_\varepsilon\left(  x_{0}\right)  \right)  =0$. Let $B:= B_1(0)$.  Applying Proposition
1.203 $iii)$ in \cite{FL},
$$\begin{array}{lll}
g(x_0)  =  \displaystyle{\lim_{\varepsilon\rightarrow 0}\frac{1}{\varepsilon^N}\frac{\mu(B_\varepsilon(x_0))}{|B|}
}\vspace{0.2cm}\\
=  \displaystyle{\limsup_{\varepsilon\rightarrow 0}\lim_{n\rightarrow +\infty}\frac{1}{\varepsilon^N\, |B|}\int_{B_\varepsilon(x_0)}f(y,u_n(y),v_n(y),\nabla u_n(y))\,dy
}\vspace{0.2cm}\\
 =  \displaystyle{\limsup_{\varepsilon\rightarrow 0}\lim_{n\rightarrow +\infty}\frac{1}{|B|}\int_Bf(x_0+\varepsilon x,u_n(x_0+\varepsilon x),v_n(x_0+\varepsilon x),\nabla u_n(x_0+\varepsilon x))\,dx}\vspace{0.2cm}\\
\ge  \displaystyle{\limsup_{\varepsilon\rightarrow 0}\lim_{n\rightarrow +\infty}\frac{1}{|B|}\int_{B}f(x_0+\varepsilon x,u(x_0)+\varepsilon w_{n,\varepsilon}(x),v_n(x_0+\varepsilon x),\nabla w_{n,\varepsilon}(x))\,dx}
\end{array}$$
where $\displaystyle{w_{n,\varepsilon}(x)=\frac{u_n(x_0+\varepsilon x)-u(x_0)}{\varepsilon}.}$
\medskip

\noindent {\bf Step 2. Blow-up.} Next we will ``identify the limits'' of $w_{n,\varepsilon}$ and $v_n(x_0+\varepsilon \cdot)$ in a sense to be made precise below. Define $w_0:B\rightarrow \mathbb{R}^n$ such that $w_0(x)=\nabla u(x_0)x$. Then
$$\begin{array}{ll}
\displaystyle{\lim_{\varepsilon\rightarrow 0}\lim_{n\rightarrow +\infty}||w_{n,\varepsilon}-w_0||_{L^1(B)}} =  \displaystyle{\lim_{\varepsilon\rightarrow 0}\lim_{n\rightarrow +\infty} \int_B\left| \frac{u_n(x_0+\varepsilon x)-u(x_0)}{\varepsilon}-\nabla u(x_0) x \right|\,dx
}\vspace{0.2cm}\\
 %= \displaystyle{\lim_{\varepsilon\rightarrow 0}\lim_{n\rightarrow +\infty} %\frac{1}{\varepsilon}\int_B\left| u_n(x_0+\varepsilon z)-u(x_0)-\nabla %u(x_0)\varepsilon z \right|\,dz} \vspace{0.2cm}\\
% =  \displaystyle{\lim_{\varepsilon\rightarrow 0} \frac{1}%{\varepsilon}\int_B\left| u(x_0+\varepsilon z)-u(x_0)-\nabla %u(x_0)\varepsilon z \right|\,dz}\vspace{0.2cm}\\
 =  \displaystyle{\lim_{\varepsilon\rightarrow 0} \frac{1}{\varepsilon^{N+1}}\int_{B_\varepsilon(x_0)}\left| u(y)-u(x_0)-\nabla u(x_0)(y-x_0) \right|\,dy=  0}
\end{array}$$
where we have used (\ref{differentiability-point}) in the last identity.

Let $q$ be the H\"{o}lder's conjugate exponent of $p$. Since $L^q$ is separable, consider $\{\varphi_l\}$ a countable dense set of functions in $L^q(B)$. Then
$$\displaystyle{\lim_{\varepsilon\rightarrow 0}\lim_{n\rightarrow +\infty} \left|\int_B(v_n(x_0+\varepsilon x)-v(x_0))\varphi_l(x)dx\right| = \lim_{\varepsilon\rightarrow 0}\left|\int_B(v(x_0+\varepsilon x)-v(x_0))\varphi_l(x)dx\right|
= 0}
$$
where we have used in the last identity  the fact that  $x_0$ is a Lebesgue point for $v$.
% up to a subsequence, $v(x_0+\varepsilon z)\rightarrow v(x_0)$, %moreover $|(v(x_0+\varepsilon z)-v(x_0))\varphi_l(z)|\le 2||v||%_{L^\infty}|\varphi_l(z)|$.
\medskip

\noindent{\bf Step 3. Diagonalization.} Arguing as in \cite{Fonseca-Muller-relaxation} and \cite{Fonseca-Kinderlehrer-Pedregal_1} we can use a diagonalization argument to find $\varepsilon_n\in\mathbb{R}^+$, $w_n\in W^{1,\infty}(\mathbb{R}^N;\mathbb{R}^n)$ and $v_n\in L^p(B;\mathbb{R}^m)\cap C^\infty_0(\mathbb{R}^N;\mathbb R^m)$, such that $\varepsilon_n \rightarrow 0$, $w_n\rightarrow w_0$ in $L^1(B;\mathbb{R}^n)$, $v_n\rightharpoonup v(x_0)$ in $L^p(B;\mathbb{R}^m)$ as $n\rightarrow +\infty$ and  $$g(x_0)\ge \lim_{n\rightarrow +\infty}\frac{1}{|B|}\int_{B}f(x_0+\varepsilon_n x,u(x_0)+\varepsilon_n w_n(x),v_n(x),\nabla w_n(x))\,dx.$$

\noindent{\bf Step 4. Truncation.}
We show that the sequences $\left\{w_n\right\}
$ and $\left\{v_n\right\}  $ constructed in the preceding steps can be replaced by sequences
$\left\{  \widetilde{w}_n\right\}  \subset W_{\text{loc}}^{1,\infty}\left(
\mathbb{R}^{N};\mathbb{R}^n\right)  $ and $\left\{  \widetilde{v}%
_n\right\}  \subset L^p(B;\mathbb R^m) \cap C^\infty_0(\mathbb{R}^N;\mathbb R^m) $ such
that $\left\Vert \widetilde{w}_n\right\Vert _{W^{1,1}\left(  B;\mathbb R^n\right)  }\leq C$, $\widetilde{w}_n\rightarrow w_{0}$ in
$L^{\infty}\left(  B;\mathbb{R}^n\right)$, \;\;\;\;\;\;\,\;\;\,\;\;\,\;
 $\left\Vert \widetilde{v}_n\right\Vert _{L^p\left(  B;\mathbb R^n\right)  }\leq
C,~\widetilde{v}_n\rightharpoonup v(x_0)$ in $L^p\left(  B;\mathbb R^m\right)  $ and
\[
g\left(  x_{0}\right)  \geq\lim_{k\rightarrow\infty}\frac{1}%
{\mathcal{L}^{N}\left(  B\right)  }\int_{B}f\left(
x_{0}+r_n x , u(x_0) +r_n\widetilde{v}_n\left(  x\right) ,
\tilde{v}_n(x), \nabla\widetilde{w}_n\left(  x\right)\right) dx.
\]
Let $0<s<t<1$ and $\lambda >1$ and define $\varphi_{s,t}$ a cut-off function such that
$0\leq\varphi_{s,t}\leq1,$ $\varphi_{s,t}\left(  \tau\right)  =1$ if $\tau\leq
s,~\varphi_{s,t}\left(  \tau\right)  =0$ if $\tau\geqslant t$ and $\left\Vert
\varphi_{s,t}^{\prime}\right\Vert _{\infty}\leq\frac{C}{
t-s}.$

Set
\begin{align*}
w_{s,t}^{n,\lambda}\left(  x\right)   &  :=w_{0}\left(  x\right)  +\varphi
_{s,t}\left(  \left\vert w_{n}\left(  x\right)  -w_{0}\left(  x\right)
\right\vert + \frac{\vert v_{n}\left(  x\right)\vert}{\lambda} \right)  \left(
w_{n}\left(  x\right)  -w_{0}\left(  x\right)  \right)  ,\\
v_{s,t}^{n,\lambda}\left(  x\right)   &  :=v(x_0)+\varphi_{s,t}\left(  \left\vert
w_{n}\left(  x\right)  -w_{0}\left(  x\right)  \right\vert +\frac{\left\vert
v_{n}\left(  x\right)  \right\vert}{\lambda} \right)  (v_{n}\left(  x\right)-v(x_0))  .
\end{align*}
Clearly,
\begin{equation}
\left\Vert w_{s,t}^{n,\lambda}-w_{0}\right\Vert _{\infty}\leq t \hbox{ and } v_{s,t}^{n,\lambda}\rightharpoonup v(x_0) \hbox{ in }L^p \hbox{ as }n \to +\infty.\label{2.9}%
\end{equation}
Define%
\[
h_{n}\left(x,  s,b, A\right)  :=f(x_0+ r_n x, u(x_0)+r_n s, b, A).
\]
By
the growth conditions there exists $n_{0}\in\mathbb{N}$ such that for all
$n\geq n_{0}$,
\begin{equation}
c\left(  | b|^p +|A|\right)  -C\leq
h_{n}\left(  x,s,b, A\right)  \leq C\left( |b|^p +|A|+1\right)  \label{2.10}%
\end{equation}
for some constants $c,~C>0.$
Consequently there exist $ C>0$ such that
$$
-C \leq h_n(x, w_0(x), v(x_0), \nabla w_0(x)) \leq C
$$

Also%
\begin{equation}\nonumber
\begin{array}{ll}
\displaystyle{\int_{B}h_{n}\left(x,   w_{s,t}^{n,\lambda}\left(  x\right),v_{s,t}^{n,\lambda}\left(  x\right)   ,\nabla w_{s,t}
^{n,\lambda}\left(  x\right)   \right)  dx
=}
\\
=\displaystyle{\int_{B\cap\left\{  \left\vert w_{n}\left(  x\right)  -w_{0}\left(
x\right)  \right\vert +\frac{| v_{n}\left(  x\right) |}{\lambda}\leq s\right\} }h_{n}\left(x,  w_{n}, v_{n}, \nabla w_{n}\right)  dx}\\
\displaystyle{+\int_{B \cap\left\{  s<\left\vert w_{n}\left(  x\right)
-w_{0}\left(  x\right)  \right\vert +\frac{| v_{n}\left(  x\right)|}{\lambda}\leq t\right\}  }h_{n}\left(  x,w_{s,t}^{n,\lambda},v_{s,t}^{n,\lambda}, \nabla w_{s,t}%
^{n,\lambda}\right)  dx}\\
\displaystyle{+\int_{B\cap\left\{  \left\vert w_{n}\left(  x\right)
-w_{0}\left(  x\right)  \right\vert +\frac{|v_{n}(  x)|}{\lambda}>t\right\}  }h_{n}\left(  x,w_{0}\left(  x\right) ,v(x_0)  ,\nabla
w_{0}\left(  x\right) \right)  dx}\\
\displaystyle{  :=I_{1}+I_{2}+I_{3}.}
\end{array}
\end{equation}
By the growth conditions and the definition of $h_{n}$%
we have that
\[
I_{3}\leq C\left\vert \left\{  x\in B:\left\vert w_{n}\left(
x\right)  -w_{0}\left(  x\right)  \right\vert +\frac{| v_{n}\left(
x\right) |}{\lambda}> t\right\}  \right\vert .
\]
On the other hand, if $s<\left\vert w_{n}\left(  x\right)  -w_{0}\left(
x\right)  \right\vert +\frac{| v_{n}\left(  x\right) |}{\lambda} <t$ then%
$$
\begin{array}{ll}
\nabla w_{s,t}^{n,\lambda}\left(  x\right)   =\nabla u\left(  x_{0}\right)
+\varphi_{s,t}\left(  \left\vert w_{n}\left(  x\right)  -w_{0}\left(
x\right)  \right\vert +\frac{| v_{n}\left(  x\right)|}{\lambda}  \right)
\left(\nabla w_{n}\left(  x\right)  -\nabla w_{0}\left(  x\right)
\right)+ \\
\left(  w_{n}\left(  x\right)  -w_{0}\left(  x\right)  \right)
\otimes\varphi'_{s,t}\left(  \left\vert w_{n}\left(  x\right)  -w_{0}\left(  x\right)
\right\vert +\frac{|v_{n}|\left(  x\right) }{\lambda} \right)
\nabla\left(  \left\vert w_{n}\left(  x\right)  -w_{0}\left(  x\right)
\right\vert +\frac{| v_{n}\left(  x\right)|}{\lambda} \right).
\end{array}
$$

%\begin{align*}
%\nabla v_{s,t}^{n}\left(  x\right)   &  =\varphi_{s,t}\left(  \left\vert
%w_{n}\left(  x\right)  -w_{0}\left(  x\right)  \right\vert +\left\vert
%v_{n}\left(  x\right)  \right\vert \right)  \nabla v_{n}\left(  x\right)  \\
%&  +v_{n}\left(  x\right)  \varphi%
%TCIMACRO{\U{b4}}%
%BeginExpansion
%\acute{}%
%EndExpansion
%_{s,t}\left(  \left\vert w_{n}\left(  x\right)  -w_{0}\left(  x\right)
%\right\vert +\left\vert v_{n}\left(  x\right)  \right\vert \right)
%\nabla\left(  \left\vert w_{n}\left(  x\right)  -w_{0}\left(  x\right)
%\right\vert +\left\vert v_{n}\left(  x\right)  \right\vert \right)  .
%\end{align*}
By $\left(  \ref{2.10}\right)  $ we have

\begin{align}
&I_{2}   \leq C\int_{B\cap\left\{  s<\left\vert w_{n}\left(
x\right)  -w_{0}\left(  x\right)  \right\vert +\frac{| v_{n}(x) |}{\lambda} \leq t\right\}  }\left(1+  \left\vert \nabla
w_{n}\left(  x\right)  -\nabla w_{0}\left(  x\right)  \right\vert +\left\vert
 v_{n}\left(  x\right) -v(x_0) \right\vert^p \right)  dx +\nonumber\\
&\frac{C}{t-s}\int_{B\cap\left\{  s<\left\vert w_{n}\left(
x\right)  -w_{0}\left(  x\right)  \right\vert +\frac{\left\vert v_{n}\left(
x\right)  \right\vert}{\lambda} \leq t\right\}  } \left\vert w_{n}\left(
x\right)  -w_{0}\left(  x\right) \right\vert|\nabla(  \left\vert w_{n}\left(  x\right)
-w_{0}\left(  x\right)  \right\vert +\frac{\left\vert v_{n}\left(  x\right)
\right\vert}{\lambda})|dx. \nonumber
\end{align}

We remark that for almost every $t$ and $\lambda$ we have
\begin{equation}
\lim_{s\rightarrow t^{-}}\int_{\left\{  s<\left\vert w_{n}\left(  x\right)
-w_{0}\left(  x\right)  \right\vert +\frac{|v_n(x)|}{\lambda}<t\right\}  }\left(  1+\left\vert \nabla
w_{n}\left(  x\right)  -\nabla u\left(  x_{0}\right)  \right\vert +\left\vert
 v_{n}\left(  x\right) -v(x_0) \right\vert^p \right)  dx=0\label{2.14}%
\end{equation}
and by the coarea formula%
\begin{equation}
\begin{array}{lll}
\displaystyle{ \lim_{s\rightarrow t^{-}}\frac{1}{t-s}\int_{B\cap\left\{
s<\left\vert w_{n}\left(  x\right)  -w_{0}\left(  x\right)  \right\vert
+\frac{\left\vert v_{n}\left(  x\right)  \right\vert}{\lambda} \leq t\right\}  }} \displaystyle{
\left\vert w_{n}\left(  x\right)  -w_{0}\left(  x\right)  \right\vert
 \cdot }\\
\;\;\;\,\,\,\;\;\;\;\;\,\,\,\;\;\;\,\;\;\;\;\;\,\,\,\,\,\,\;\,\;\;\,\,\,\;\;\displaystyle{\left\vert\nabla\left(
\left\vert w_{n}\left(  x\right)  -w_{0}\left(  x\right)  \right\vert
+\frac{\left\vert v_{n}\left(  x\right)  \right\vert}{\lambda} \right)\right\vert dx \leq} \label{coareaFM}\\
\le\displaystyle{ \lim_{s\rightarrow t^{-}}\frac{1}{t-s}\int_{B\cap\left\{
s<\left\vert w_{n}\left(  x\right)  -w_{0}\left(  x\right)  \right\vert
+\frac{\left\vert v_{n}\left(  x\right)  \right\vert}{\lambda} \leq t\right\}  }} \displaystyle{\left(
\left\vert w_{n}\left(  x\right)  -w_{0}\left(  x\right)  \right\vert
+\frac{\left\vert v_{n}\left(  x\right)  \right\vert}{\lambda} \right)}\\
\\
\;\;\;\,\,\,\;\;\;\;\;\,\,\,\;\;\;\,\;\;\;\;\;\,\,\,\,\,\,\;\,\;\;\,\,\,\;\displaystyle{\left|  \nabla\left(
\left\vert w_{n}\left(  x\right)  -w_{0}\left(  x\right)  \right\vert
+\frac{\left\vert v_{n}\left(  x\right)  \right\vert}{\lambda} \right)\right| dx}\\
 \displaystyle{\leq t\mathcal{H}^{N-1}\left(  \left\{  x\in B:\left\vert
w_{n}\left(  x\right)  -w_{0}\left(  x\right)  \right\vert +\frac{\left\vert
v_{n}\left(  x\right)  \right\vert}{\lambda} =t\right\}  \right) .}
\end{array}
\end{equation}
Due to the fact that $\{v_n\}$ is a $C^\infty_0(\mathbb{R}^N;\mathbb{R}^m)$ sequence, for every $C>0$ , for every $n$ there exists $\lambda_n \in [1,+\infty)$ such that $\lambda_n \leq \lambda_{n+1}$, $\lambda_n \to +\infty$ as $n \to +\infty$ and
$\displaystyle{\int_B \frac{|\nabla |v_n||}{\lambda_n}dx \leq C}$. On the other hand by
$\left(  \ref{2.10}\right)  $%
\begin{align*}
&  \int_{B\cap\left\{  \left\vert w_{n}\left(  x\right)
-w_{0}\left(  x\right)  \right\vert +\frac{\left\vert v_{n}\left(  x\right)
\right\vert}{\lambda_n} \leq1\right\}  }\left\vert \nabla\left(  \left\vert w_{n}\left(
x\right)  -w_{0}\left(  x\right)  \right\vert\right)  \right\vert dx\\
&  \leq\int_{B\cap\left\{  \left\vert w_{n}\left(  x\right)
-w_{0}\left(  x\right)  \right\vert +\frac{\left\vert v_{n}\left(  x\right)
\right\vert}{\lambda_n} \leq1\right\}  }\left(  \left\vert \nabla w_{n}\left(  x\right)
\right\vert +C\right)
dx\\
&  \leq C\int_{B}h_{n}\left( x, u_n, \left(  x\right), v_{n}\left(  x\right) ,\nabla
w_{n}\left(  x\right)   \right)  dx\leq C
\end{align*}
since $\left\{  w_{n}\right\}  $ and $\left\{  v_{n}\right\}  $ are
convergent.

Thus
$$
\int_{B\cap\left\{  \left\vert w_{n}\left(  x\right)
-w_{0}\left(  x\right)  \right\vert +\frac{\left\vert v_{n}\left(  x\right)
\right\vert}{\lambda_n} \leq1\right\}  }\left\vert \nabla\left(  \left\vert w_{n}\left(
x\right)  -w_{0}\left(  x\right)  \right\vert +\frac{\left\vert v_{n}\left(
x\right)  \right\vert}{\lambda_n} \right)  \right\vert dx \leq C.
$$
Recall  that $\displaystyle{\int_B \frac{|v_n|^p}{\lambda_n} dx\leq C}$ and by H\"{o}lder's inequality also $\displaystyle{\int_B \frac{|v_n|}{\lambda_n} dx\leq C}$.
Hence, by Lemma 2.6 in \cite{Fonseca-Muller-quasiconvex} there exists \newline$t_{n}%
\in\left[  \left(  \left\Vert w_{n}-w_{0}\right\Vert _{L^{1}}+\frac{\left\Vert
v_{n}\right\Vert _{L^{1}}}{\lambda_n}\right)  ^{\frac{1}{2}},\left\Vert \left(  \left\Vert
w_{n}-w_{0}\right\Vert _{L^{1}}+\frac{\left\Vert v_{n}\right\Vert _{L^{1}}}{\lambda_n}\right)
^{\frac{1}{3}}\right\Vert \right]  $ such that $\left(  \ref{2.14}\right)  $
and $\left(  \ref{coareaFM}\right)  $ hold (with $t=t_{n}$), and%
\[
t_{n}\mathcal{H}^{N-1}( \{  x\in B
%TCIMACRO{\U{b4}}%
%BeginExpansion
%
%EndExpansion
:\left\vert w_{n}\left(  x\right)  -w_{0}\left(  x\right)  \right\vert
+\frac{\left\vert v_{n}\left(  x\right)  \right\vert}{\lambda_n} =t_{n}\} )
\leq\frac{C}{\ln\left(  \left\Vert w_{n}-w_{0}\right\Vert _{L^{1}%
}+\frac{\left\Vert v_{n}\right\Vert _{L^{1}}}{\lambda_n}\right)  ^{-\frac{1}{6}}}.
\]

According to $\left(  \ref{2.14}\right)  $ and $\left(  \ref{coareaFM}\right)
$ we may choose $0<s_{n}<t_{n}$ such that
\[
\int_{\left\{  s_{n}<\left\vert w_{n}\left(  x\right)  -w_{0}\left(  x\right)
\right\vert +\frac{|v_n(x)|}{\lambda_n}\leq t_{n}\right\}  }\left(  1+\left\vert \nabla w_{n}\left(
x\right)  -\nabla u\left(  x_{0}\right)  \right\vert +\left\vert
v_{n}\left(  x\right) -v(x_0) \right\vert^p \right)  dx=O(  \frac{1}{n})  ,
\]%
$$
\begin{array}{ll}
\displaystyle{\frac{1}{t_{n}-s_{n}}\int_{B\cap\left\{  s_{n}<\left\vert
w_{n}\left(  x\right)  -w_{0}\left(  x\right)  \right\vert +\frac{\left\vert
v_{n}\left(  x\right)  \right\vert}{\lambda_n} \leq t_{n}\right\}  }\left(  \left\vert
w_{n}\left(  x\right)  -w_{0}\left(  x\right)  \right\vert +\frac{\left\vert
v_{n}\left(  x\right)  \right\vert}{\lambda_n} \right)\cdot}\\
\;\;\,\,\;\;\,\;\;\;\,\;\;\;\;\;\,\,\,\,\,\;\;\;\;\;\;\,\,\;\;\;\,\,\,\,\;\;\;\,\,\;\;\displaystyle{\left|  \nabla\left(  \left\vert
w_{n}\left(  x\right)  -w_{0}\left(  x\right)  \right\vert +\frac{\left\vert
v_{n}\left(  x\right)  \right\vert}{\lambda_n} \right)\right| dx} \\
\displaystyle{\leq t_{n}\mathcal{H}^{N-1}\left(  \left\{x \in B:
\left\vert w_{n}\left(  x\right)  -w_{0}\left(  x\right)  \right\vert
+\frac{\left\vert v_{n}\left(  x\right)  \right\vert}{\lambda_n} =t_{n}\right\}  \right)
+O\left(  \frac{1}{n}\right)}
\end{array}
$$
Set
\[
\widetilde{w}_{n}\left(  x\right)  :=w_{s_{n},t_{n}}^{n,\lambda_n}\left(  x\right)
,\qquad\widetilde{v}_{n}\left(  x\right)  :=v_{s_{n},t_{n}}^{n,\lambda_n}\left(
x\right)
\]
thus by \eqref{2.9}
\[
\left\Vert \widetilde{w}_{n}-w_{0}\right\Vert_\infty \leq t_{n}\rightarrow
0,\qquad \widetilde{v}_{n}\rightharpoonup v(x_0) \hbox{ in } L^p \text{ as }n\rightarrow\infty.
\]
Using the previous estimates we conclude that
\begin{align*}
g\left(  x_{0}\right)   &  \geq\lim_{n\rightarrow\infty}\frac
{1}{\mathcal{L}^{N} ( B)  }\int_B f\left( x_{0} + r_n x,u(x_0)+r_{n} w_n(x),  v_n(x), \nabla w_n( x)  \right) dx\\
&  \geq\underset{n\rightarrow\infty}{\lim\inf}\frac{1}{\mathcal{L}^{N}\left(
B\right)  }\int_{B\cap\left\{  \left\vert w_n\left(
x\right)  -w_{0}\left(  x\right)  \right\vert +\frac{\left\vert v_n\left(
x\right)  \right\vert}{\lambda_n} \leq s\right\}  }h_{n}\left(x,  w_n\left(  x\right)
v_n(x),\nabla w_{n}\left(  x\right)  \right)  dx\\
&  \geq\underset{n\rightarrow\infty}{\lim\inf}\frac{1}{\mathcal{L}^{N}\left(
B\right)  }\int_B h_{n}\left(x,  \widetilde{w}_{n}\left(  x\right)
,\widetilde{v}_{n}\left(
x\right),\nabla \widetilde {w}_n\left(  x\right) \right)  dx-O\left(  \frac{1}{n}\right)+
\\
&-\frac{C}{\ln\left(
\left\Vert w_{n}-w_{0}\right\Vert _{L^{1}\left(  B\right)
}+\frac{\left\Vert v_{n}\right\Vert _{L^{1}\left(  B^{\prime}\right) }}{\lambda_n}\right)
^{-\frac{1}{6}}}\\
&  -C\mathcal{L}^{N}\left(  \left\{  x\in B:\left\vert w_{n}\left(  x\right)
-w_{0}\left(  x\right)  \right\vert +\frac{\left\vert v_{n}\left(  x\right)
\right\vert}{\lambda_n} >t_{n}\right\}  \right) \\
&  =\underset{n\rightarrow\infty}{\lim\inf}\frac{1}{\mathcal{L}^{N}\left(
B\right)  }\int_{B}h_{n}\left(x,  \widetilde{w}_{n}\left(  x\right)
,\widetilde{v}_{n}\left(
x\right) , \nabla\widetilde{w}_{n}\left(  x\right) \right)  dx,
\end{align*}
since
\[
t_{n}\geq\left(  \left\Vert w_{n}-w_{0}\right\Vert _{L^{1}(B)  }+\frac{|| v_n||_{L^{1}(B)}}{\lambda_n} \right) ^{\frac{1}{2}}%
\]
and thus%
\begin{align*}
\mathcal{L}^{N}\!\!\!\left(  \left\{  x\in B:\left\vert w_{n}\left(  x\right)
-w_{0}\left(  x\right)  \right\vert \!+\!\!\frac{\left\vert v_{n}\left(  x\right)
\right\vert}{\lambda_n} >t_{n}\right\}  \right) \!\!\!&\leq\frac{1}{t_{n}}\left(  \left\Vert
w_{n}-w_{0}\right\Vert _{L^{1}(B) }+\frac{\left\Vert v_{n}\right\Vert _{L^{1}(B)}}{\lambda_n}  \right) \\
&  \!\!\!\!\!\!\leq\left(  \left\Vert w_{n}-w_{0}\right\Vert _{L^{1}(B)  }+\frac{\left\Vert v_{n}\right\Vert _{L^{1}(B)}}{\lambda_n}\right)  ^{\frac{1}{2}}\rightarrow0.
\end{align*}
The bound of $\left\{  \left\Vert \nabla\widetilde{w}_{n}\right\Vert _{L^{1}%
}\right\}  $ %and $\left\{  \left\Vert \nabla\widetilde{v}_{n}\right\Vert
%_{L^{1}}\right\}  $
follows from $\left(  \ref{2.10}\right)  .$

\noindent{\bf Step 5.} We now fix in $f$ the value of $x$ and $u$. Indeed, using hypothesis $(H2_p)$ and the fact that $\nabla \tilde{w}_n$ and $|\tilde{v}_n|^p$ have bounded $L^1$ norm, one gets
$$\begin{array}{rcl}g(x_0) & \ge & \displaystyle{\limsup_{n\rightarrow +\infty}\frac{1}{|B|}\int_{B}f(x_0+\varepsilon_n x,u(x_0)+\varepsilon_n \tilde{w}_n(x),\tilde{v}_n(x),\nabla \tilde{w}_{n}(x))\,dx}\vspace{0.2cm}\\& \ge & \displaystyle{\limsup_{n\rightarrow +\infty}\frac{1}{|B|}\int_{B}f(x_0,u(x_0),\tilde{v}_n(x),\nabla \tilde{w}_{n}(x))\,dx}.
\end{array}$$

\noindent{\bf Step 6.} At this point we are in an analogous  context to \cite{Fonseca-Kinderlehrer-Pedregal_1} and the desired inequality follows in the same way. It relies on the slicing method in order to modify $\tilde{v}_n$ and $\tilde{w}_n$ and exploit the convex-quasiconvexity of $f$, namely it is possible to find new sequences, denoted by $\bar{v}_n$ and $\bar{w}_n$ such that
$$\frac{1}{|B|}\int_{B}\bar{v}_n(z)\,dz=v(x_0)\text{ and }\bar{w}_j\in w_0+W_0^{1,\infty}(B;\mathbb{R}^n).$$ \end{proof}

\section{Relaxation in $W^{1,1}\times L^\infty$}\label{CaseLinfty}

This section is devoted to  characterize the  relaxed functional $\overline{J}_\infty$ introduced in \eqref{Jinftybar}.

Indeed we prove the following relaxation result

\begin{Theorem}\label{maintheoreminftyrelax}
Let $\Omega $ be a bounded open set of $\mathbb R^N$, and let $f:\Omega\times \mathbb{R}^n \times \mathbb{R}^m\times \mathbb{R}^{n\times N}\rightarrow [0,+\infty)$ be a continuous function. Then, assuming that $f$ and $CQf$ satisfy hypotheses $(\mathrm{H}1_\infty)$ and $(\mathrm{H}2_\infty)$
$$
\begin{array}{rcl}\overline{J}_\infty(u,v) & = & \displaystyle{\int_\O CQf(x,u(x),v(x),\nabla u(x))\,dx}, \end{array}
$$
for every $(u,v)\in W^{1,1}(\Omega;\mathbb R^n)\times L^\infty(\Omega;\mathbb R^m)$.
\end{Theorem}

\begin{Remark}
\noindent 1) We recall that if hypotheses $(H1_\infty)$ and $(H2_\infty)$ are replaced by \eqref{H1_1inftyCQf}$-$\eqref{H2_2inftyCQf}, Propositions \ref{measCQf_infty} and \ref{FMnonquasiconvex} guarantee the validity of Theorem \ref{maintheoreminftyrelax} assuming that only $f$ satisfies \eqref{H1_1inftyCQf}$-$\eqref{H2_2inftyCQf}.

\noindent 2) We also observe that Theorem \ref{maintheoreminftyrelax} can be proven also imposing $(H1_\infty)$ and $(H2_\infty)$ only on the function $f$ but with the further requirement that $f$ satisfies \eqref{MarcelliniLipschitz}.

\noindent 3) We also stress that if $f$ satisfies $(H1_p)$ and $(H2_p)$ then clearly $\overline{J}_p(u,v) \leq \overline{J}_\infty(u,v)$ for every $(u,v) \in BV(\Omega;\mathbb R^N) \times L^\infty(\Omega;\mathbb R^m)$.
\end{Remark}
\begin{proof}[Proof of Theorem \ref{maintheoreminftyrelax}]
The thesis will be achieved by double inequality. Clearly the lower bound can be proven as for the case $W^{1,1}\times L^p$, with a proof easier than that of Theorem \ref{maintheoremp}, since it is not necessary 'truncate' the $\{v_n\}$ which are already bounded in $L^\infty$.
 For what concerns the upper  bound, we first observe that by virtue of Proposition \ref{FMnonquasiconvex}, there is no loss of generality in assuming $f$ already convex-quasiconvex.
In order to provide an upper bound for $\overline{J}_\infty$ we start by localizing our functional. The following procedure is entirely similar to \cite[Theorem 4.3]{AMT}.
We define for every open set $A \subset \Omega$ and for any $(u,v)\in BV(\Omega;\mathbb R^n)\times L^\infty(\Omega;\mathbb R^m)$

$$
\displaystyle{
\overline{F}_\infty(u,v, A):=\inf \left\{ \liminf_n F(u_n,v_n,A): u_n\to u \hbox{ in }L^1(A;\mathbb R^n), v_n \weakstar v \hbox{ in }L^\infty(A;\mathbb R^m)\right\}}$$
where
$$F(u,v,A)=\left\{\begin{array}{l}\displaystyle{\int_A f(x,u(x),v(x),\nabla u(x))\,dx} \,\text{if}\ (u,v)\in W^{1,1}(A;\mathbb{R}^n)\times L^\infty(A;\mathbb{R}^m),\vspace{0.2cm}\\ +\infty \ \text{ in }(L^1(A;\mathbb R^n)\setminus W^{1,1}(A;\mathbb R^n)) \times L^\infty(A;\mathbb R^m).\end{array}\right.$$

We start remarking that $(H1_\infty)$ implies that for every $u \in BV(\Omega;\mathbb{R}^n)$ and for every $v \in L^\infty(\Omega;\mathbb{R}^m)$ such that $\|v\|_{L^\infty} \leq M$, there exists a constant $C_M$ such that $\displaystyle{\overline{F}_\infty(u,v,A)\le C_M(|A|+|Du|(A)).}$ Moreover one has

\noindent 1) $\overline{F}_\infty$ is local, i.e. $\overline{F}_\infty(u,v,A)= \overline{F}_\infty(u',v',A)$, for every $A \subset \Omega$ open, $(u,v), (u',v') \in L^1(A;\mathbb R^n)\times L^\infty(A;\mathbb R^m)$.

\noindent 2) $\overline{F}_\infty$ is sequentially lower semi-continuous, i.e.
$\overline{F}_\infty(u,v,A)\le\liminf\overline{F}_\infty(u_n,v_n,A),$ $ \forall\ A\subset\Omega \text{ open, }\forall\ u_n\to u \text{ in } L^1(A;\mathbb{R}^n)\text{ and }v_n\weakstar v\text{ in }L^\infty(A;\mathbb{R}^m);$

\noindent 3) $\overline{F}_\infty(u,v,\cdot)$ is the trace on ${\cal A}(\Omega):=\{A\subset \Omega:\ A\text{ is open}\}$ of a Borel measure in $\mathcal{B}(\Omega)$ (the Borelians of $\Omega$).
\medskip

Condition 1) follows from the fact the adopted convergence doesn't see sets of null Lebesgue measure.
Condition 2) follows by a diagonalization argument, entirely similar to the proof of (ii) in \cite{Fonseca-Kinderlehrer-Pedregal_1}.
Condition 3) follows applying De Giorgi-Letta criterium, (cf. \cite{DeGiorgi-Letta}) and indeed proving that for any fixed $(u,v)\in BV(\Omega;\mathbb{R}^n)\times L^\infty(\Omega;\mathbb{R}^m)$,
\begin{equation}\nonumber \overline{F}_\infty(u,v,A)\le \overline{F}_\infty(u,v,C)+\overline{F}_\infty(u,v,A\setminus \overline{B}),\ \forall\ A,B,C\in {\cal A}(\Omega).
\end{equation}
We omit the details, since they are very similar to the proof of Theorem 4.3 in \cite{AMT}. The only difference consists of the fact that one has to deal with both $u's$ and $v's$ and exploit the growth condition $(H1_\infty)$.

Since $\overline{J}_\infty(u,v) = \overline{F}_\infty(u,v, \Omega)$ and $\overline{F}_\infty(u,v, \cdot)$ is the trace of a Radon measure on the open subsets of $\Omega$, (i. e. ${\cal A}(\Omega)$) absolutely continuous with respect to $|Du| + {\cal L}^N$, it will be enough to prove the following inequality
\begin{equation}\nonumber
\frac{d \overline{F}_\infty(u,v,\cdot)}{d {\cal L}^N}(x)\leq f(x,u(x),v(x),\nabla u(x)),\ \mathcal{L}^N-a.e.\ x\in\O.\end{equation}

The proof of these inequalities follows closely \cite{Fonseca-Muller-relaxation}, \cite{AMT} and \cite{BZZ}.

\noindent Assume first that $(u,v)\in (W^{1,1}(\Omega;\mathbb R^n) \cap L^\infty(\Omega;\mathbb R^n)) \times L^\infty(\Omega;\mathbb R^m)$. Fix a point $x_0 \in \Omega$ such that
\begin{equation}\label{LebPointsUB}
\frac{d \overline{F}_\infty(u,v,\cdot)}{d {\cal L}^N}(x_0)
\end{equation} exists and is finite, which is also a Lebesgue
point of $u$, $v$ and $\nabla u$ and a point
of approximate differentiability for $u$.
%and such that
%\begin{equation}\label{SingUB}
%\frac{d|D^s u|}{d{\cal L}^N}(x_0)=0.
%\end{equation}
%Observe that since ${\cal L}^N$ is singular with respect to $|D^s
%u|$, then
Clearly ${\cal L}^N$-a.e. $x_0\in \Omega$ satisfy all the above requirements.

\noindent As in \cite{Fonseca-Muller-relaxation} (see formula (5.6) therein) we may also assume that
\begin{equation}\label{FMr5.6}
\lim_{\e \to 0}\frac{1}{|Q(x_0,\e)|}\int_{Q(x_0,\e)}|u(x)-u(x_0)|(1+ |\nabla u(x)|)dx=0,
\end{equation}
\begin{equation}\label{FMr5.6bis}
\lim_{\e \to 0}\frac{1}{|Q(x_0,\e)|}\int_{Q(x_0,\e)}|v(x)-v(x_0)||\nabla u(x)|dx=0,
\end{equation}
(where we used Theorem \ref{FMrthm2.8} since $v \in L^1_{loc}(\Omega;\mathbb R^m)$ with respect to the measure $|\nabla u|\mathcal{L}^N$).  Choose a sequence of numbers $\e \in (0, {\rm dist} (x_0,  \partial \Omega))$. Then, clearly for any sequences $\{u_n\}$, $u_n \to u$ in $L^1$ , $\{v_n\}$, $v_n \weakstar v$ in $L^\infty$,
 \begin{equation}\label{FMr5.9}
 \begin{array}{lll}
\displaystyle{ \frac{\partial \overline{F}_\infty(u,v, \cdot)}{\partial {\cal L}^N}(x_0)= \lim_{\e \to 0^+}\frac{\overline{F}_\infty(u,v, B_\e(x_0))}{|B_\e(x_0)|}\leq}\\
 \\
 \displaystyle{\liminf_{\e \to 0^+} \liminf_{n \to + \infty} \frac{1}{|B_\e(x_0)|}\int_{B_\e(x_0)}f(x,u_n(x),v_n(x),\nabla u_n(x))dx.}
 \end{array}
 \end{equation}

 By virtue of Proposition 2.2 in \cite{Ambrosio-Dal Maso} we can replace  the ball $B_\e(x_0)$  in (\ref{FMr5.9}) by a cube of side length $\e$, and in fact from now on we consider such cubes.

 As in Proposition 4.6 of \cite{AMT}, (see also \cite{Fonseca-Muller-relaxation} and \cite{Fonseca-Kinderlehrer-Pedregal_2}) we consider the Yosida transforms of $f$, defined as

$$ \displaystyle{f_\lambda(x,u,v,\xi):=\sup_{(x',u')\in \Omega \times \mathbb R^n}\{f(x',u',v,\xi)-\lambda[|x-x'|+|u-u'|](1 + |\xi|+|v|) \}}$$
for every $\lambda >0$.
 Then

 \noindent $(i)$ $f_\lambda(x,u,v,\xi)\geq f(x,u,v,\xi)$ and $f_\lambda (x,u,v,\xi)$ decreases to $f(x,u,v,\xi)$ as $\lambda \to + \infty$.

 \noindent $(ii)$ $f_\lambda(x,u,v,\xi) \geq f_\eta(x,u,v,\xi)$ if $\lambda \leq \eta$ for every $(x,u,v,\xi)\in \Omega \times \mathbb R^n \times \mathbb R^m \times \mathbb R^{n \times N}$.

 \noindent$(iii)$ $|f_\lambda (x,u,v,\xi)-f_\lambda(x',u',v,\xi)| \leq \lambda (|x-x'|+|u-u'|)(1+ |\xi| + |v|)$ for every $(x,u,v,\xi), $ $(x',u',v,\xi)\in \Omega \times \mathbb R^n \times \mathbb R^m \times \mathbb R^{n \times N}$.

 \noindent $(iv)$ The approximation is uniform on compact sets. Precisely let $K$ be a compact subset of $\Omega \times \mathbb R^n$ and let $\delta>0$. There exists $\lambda>0 $ such that

$\displaystyle{f(x,u,v,\xi)\leq f_\lambda(x,u,v,\xi)\leq f(x,u,v,\xi) + \delta (1+ |v|+ |\xi|)}$
 for every $(x, u, v, \xi) \in K \times \mathbb R^m \times \mathbb R^{n \times N}$.

Let $x_0$ such that \eqref{LebPointsUB} and \eqref{FMr5.6} hold,
let $\{\varrho_n\} $ be a sequence of standard symmetric mollifiers and set
$ \left\{\begin{array}{ll}
u_n:= u \ast \varrho_n,\\
v_n:= v
\end{array}
\right.$.
It results that
$
\left\{
\begin{array}{ll}
u_n \to u &\hbox{ in }L^1(Q(x_0,\e);\mathbb R^n),\\
v_n \weakstar v &\hbox{ in }L^\infty(Q(x_0, \e);\mathbb R^m).
\end{array}
\right.
$
Fix $\delta >0$ and let $K:= \overline{B}(x_0, \frac{{\rm dist}(x_0, \partial \Omega)}{2}) \times \overline{B}(0, \| u\|_\infty)$. By $(i)\div (iv)$,
$$
\begin{array}{ll}
f(x,u_n(x), v(x), \nabla u_n(x)) \leq f_\lambda(x, u_n(x), v(x), \nabla u_n(x)) \leq\\
\\
f_\lambda(x_0, u(x_0), v_n(x), \nabla u_n(x))+ \lambda(|x-x_0|+ |u_n(x)-u(x_0)|)(1+ |v|+ |\nabla u_n(x)|)\leq\\
\\
f(x_0,u(x_0), v(x), \nabla u_n(x))+ \delta(1 + |\nabla u_n(x)|+ |v(x)|) +\lambda (|x-x_0|+ |u_n(x)-u(x_0)|)\\
\\
\times (1 + |\nabla u_n(x)|+ |v(x)|).
\end{array}
$$
Since $\nabla u_n (x)= (  \nabla u \ast\varrho_n  )(x)$,
\begin{equation}\label{upperineq1}
\begin{array}{l}
\displaystyle{\overline{F}_\infty(u,v, Q(x_0,\e))\leq \liminf_{n \to +\infty}\int_{Q(x_0,\e)}f(x,u_n(x), v(x),\nabla u_n(x))dx \leq}\\
\\
\displaystyle{\liminf_{n \to + \infty}\int_{Q(x_0,\e)}f(x_0,u(x_0), v(x), \nabla u_n(x))dx +}\\
\\
\displaystyle{\limsup_{n \to + \infty}\int_{Q(x_0,\e)} \delta(1 + |\nabla u_n(x)|+ |v(x)|) +\lambda (|x-x_0|+ |u_n(x)-u(x_0)|)}\\
\displaystyle{\;\;\;\;\;\;\;\,\,\;\;\;\;\;\;\;\;\;\;\,\;\;\;\times (1 + |\nabla u_n(x)|+ |v(x)|)dx \leq}\\
\\
\displaystyle{\liminf_{n \to +\infty}\int_{Q(x_0,\e)} f\left(x_0, u(x_0), v(x_0), \nabla u(x_0)\right)dx +}\\
\\
\displaystyle{\limsup_{n \to + \infty} \int_{Q(x_0,\e)} \beta \left(1+ \left|\nabla u(x_0) + \right|+ \left|\nabla u \ast \varrho_n\right|\right)|v(x)-v(x_0)| dx +}\\
\\
\displaystyle{\limsup_{n \to +\infty}\int_{Q(x_0,\e)} \beta | \nabla u \ast \varrho_n - \nabla u(x_0)| dx+ }\\
\\
\displaystyle{\limsup_{n \to + \infty}\int_{Q(x_0,\e)} \delta(1 + |\nabla u_n(x)|+ |v(x)|) +\lambda (|x-x_0|+ |u_n(x)-u(x_0)|)}\\
\displaystyle{\;\;\;\;\;\;\;\;\;\;\;\;\,\;\;\;
\times (1 + |\nabla u_n(x)|+ |v(x)|)dx}
\end{array}
\end{equation}
(where the constant $\beta$, depending on $\|v\|_{L^\infty}$, is the constant appearing in (\ref{MarcelliniLipschitz})).

\noindent Since $\nabla u \ast \varrho_n \to \nabla u  \in L^1_{\rm loc}(\Omega;\mathbb R^{n \times N})$  and $|D^s u|(\partial Q(x_0,\e))=0$ for each $\e >0$, we obtain
\begin{equation}\label{upperineq2}
\begin{array}{ll}
\displaystyle{\limsup_{n \to +\infty}\int_{Q(x_0,\e)} \beta | \nabla u \ast \varrho_n - \nabla u(x_0)| dx \leq \beta \int_{Q(x_0,\e)}|\nabla u(x)- \nabla u(x_0)|dx.}
\end{array}
\end{equation}

\noindent Passing to the limit on the right hand side of (\ref{upperineq1}), exploiting \eqref{upperineq2} in the third line and applying \cite[Lemma 2.5]{Fonseca-Muller-relaxation}, in the fourth line, we get
$$
\begin{array}{ll}
\displaystyle{\overline{F}_\infty(u,v, Q(x_0,\e))\leq |Q(x_0, \e)|(f(x_0,u(x_0), v(x_0), \nabla u(x_0))) +}\\
\\
\displaystyle{\beta (1 + |\nabla u(x_0)|) \int_{Q(x_0, \e)}|v(x)-v(x_0)|dx+}\\
\\
\displaystyle{\beta \limsup_{n \to + \infty}\int_{Q(x_0,\e)}|\nabla u \ast \varrho_n||v(x)-v(x_0)|dx+ \int_{Q(x_0,\e)}|\nabla u(x)- \nabla u(x_0)|dx +}\\
\\
\displaystyle{(\lambda \e +\delta)\left[(1+ C)|Q(x_0,\e)| \right]+ \lambda \limsup_{n \to + \infty}\int_{Q(x_0,\e)}|u_n- u(x_0)|(1+ C +|\nabla u_n|)dx.}
\end{array}
$$

\noindent Recalling that $x_0$ is a Lebesgue point for $v$, $\nabla u$ and \eqref{LebPointsUB} holds, we have
$$
\begin{array}{ll}
\displaystyle{\limsup_{\e \to 0^+}\frac{1}{|Q(x_0,\e)|}\beta (1 + |\nabla u(x_0)|) \int_{Q(x_0, \e)}|v(x)-v(x_0)|dx=0,}\\
\\
\displaystyle{\limsup_{\e \to 0^+}\frac{1}{|Q(x_0,\e)|}\int_{Q(x_0,\e)}|\nabla u(x)- \nabla u(x_0)|dx=0,}\\
\\
\displaystyle{\limsup_{\e \to 0^+}(\lambda \e + \delta)(1+ C)\frac{|Q(x_0,\e)|}{|Q(x_0,\e)|}=\delta (1+C)}.
\end{array}
$$

\noindent Moreover by virtue of (\ref{FMr5.6}) and arguing as in the estimate of formula (5.11) of \cite{Fonseca-Muller-relaxation} we can conclude that
$$
\displaystyle{\limsup_{\e \to 0^+}\frac{\lambda}{|Q(x_0,\e)|}\limsup_{n \to +\infty}\int_{Q(x_0,\e)}|u_n- u(x_0)|(1+ C +|\nabla u_n|)dx=0.}
$$

\noindent Then we can  exploit  (\ref{FMr5.6bis}) and argue again as done for (5.11) in \cite{Fonseca-Muller-relaxation} in order to evaluate
$$
\displaystyle{\limsup_{\e \to 0^+}\frac{\beta}{|Q(x_0,\e)|} \limsup_{n \to + \infty} \int_{Q(x_0,\e)} | \nabla u \ast \varrho_n||v(x)-v(x_0)|dx.}
$$
\noindent We will apply \cite[Lemma 2.5]{Fonseca-Muller-relaxation} and the dominated convergence theorem with respect to the measure $|\nabla u|dx$, obtaining
$$
\begin{array}{ll}
\displaystyle{\limsup_{n \to +\infty}\int_{Q(x_0,\e)}|v(x)-v(x_0)||\nabla u_n(x)|dx \leq }
\\
\\
\displaystyle{\limsup_{n \to + \infty} \int_{Q\left(x_0, \e +\frac{1}{n}\right)} (|v-v(x_0)|\ast \varrho_n)|\nabla u (x)|dx\leq}\\
\\
\displaystyle{\int_{\overline{Q(x_0, \e)}} |v(x)-v(x_0)|
|\nabla u(x)|dx .}
%\\
%\\
%\displaystyle{\int_{\overline{Q(x_0, \e)} \setminus J_u} |v(x)-v(x_0)|
%|Du(x)| + 2 \|v\|_{L^\infty}|D^s u|(\overline{Q(x_0, \e )})}.
\end{array}
$$
Taking into account that $|Du|(\partial Q(x_0,\e))=0$ for a.e. $\e$ %and %that
%$$
%\begin{array}{ll}
%\int_{Q(x_0,\e)}|v(x)- v(x_0)| |Du|(x)\leq \int_{Q(x_0,\e)}|v(x)-v(x_0)||%\nabla u (x)|dx + 2 \|v\|_{L^\infty}|D^s u|(Q(x_0,\e)),
%\end{array}
%$$
one obtains from (\ref{FMr5.6}) that
$$
\displaystyle{\limsup_{\e \to 0^+}\limsup_{n \to +\infty}\frac{1}{|Q(x_0,\e)|}\int_{Q(x_0,\e)}|v(x)-v(x_0)||\nabla u_n(x)|dx=0.}
$$
Consequently,
$$
\begin{array}{ll}
\displaystyle{g(x_0)=\frac{\partial{\overline F}_\infty(u,v)(x_0)}{\partial{\cal L}^N}\leq f(x_0,u(x_0), v(x_0), \nabla u(x_0)) + (1+C)\delta}
\end{array}
$$
\nonumber Finally, we send  $\delta$ to $0$ and that concludes the proof, when $(u,v) \in (W^{1,1}(\Omega;\mathbb R^n)\cap L^\infty(\Omega;\mathbb R^n))\times L^\infty(\Omega;\mathbb R^m)$.

To conclude the proof, we can argue as in \cite[Theorem 2.16, Step 4]{Fonseca-Muller-relaxation}, in turn inspired by \cite{AMT}, introducing the following approximation.

\noindent Let  $\phi_n \in C^1_0(\mathbb R^n;\mathbb R^n)$ be such that $\phi_n(y)=y \hbox{ if } y \in B_n(0), \;\; \|\nabla \phi_n\|_{L^\infty}\leq 1.$
By \cite[Theorem 3.96]{AFP} $\phi_n(u)\in W^{1,1}(\Omega;\mathbb R^n)\cap L^\infty(\Omega;\mathbb R^n) $ for every $n \in \mathbb N$.
\noindent Since $\phi_n(u) \to u$ in $L^1$, by the lower semicontinuity of $\overline{J}_\infty$ we get
$$
\displaystyle{\overline{J}_\infty(u,v)\leq \liminf_{n \to + \infty}\int_\Omega f(x, \phi_n(u),v, \nabla \phi_n(u))\,dx.}
$$
Arguing in analogy with \cite[Theorem 4.9]{AMT} one can prove that
$$
\displaystyle{\limsup_{n \to +\infty} \int_\Omega f(x, \phi_n(u), v, \nabla \phi_n(u))\,dx \leq \int_\Omega f(x, u, v,\nabla u)\,dx,}
$$
and this concludes the proof.
\end{proof}

\section{Relaxation in $W^{1,1}\times L^p$}
This section is devoted to the proof of the following theorem.
It relies on Theorem \ref{maintheoreminftyrelax} and on some approximation results (see \cite{AMT}).
\begin{Theorem}\label{maintheoremprelax}
Let $\Omega $ be a bounded open set of $\mathbb R^N$, and let $f:\Omega\times \mathbb{R}^n \times \mathbb{R}^m\times \mathbb{R}^{n\times N}\rightarrow [0,+\infty)$ be a continuous function. Then, assuming that $f$ satisfies hypotheses $(\mathrm{H}1_p)$ and $(\mathrm{H}2_p)$
$$
\begin{array}{rcl}\overline{J}_p(u,v) & = & \displaystyle{\int_\O CQf(x,u(x),v(x),\nabla u(x))\,dx}, \end{array}
$$
for every $(u,v)\in W^{1,1}(\Omega;\mathbb R^n)\times L^p(\Omega;\mathbb R^m)$.
\end{Theorem}
 \begin{proof}[Proof]
The lower bound follows from Theorem \ref{maintheoremp}.
For what concerns the upper bound, without loss of generality, by virtue of Lemma \ref{FMnonquasiconvex} and Propostion \ref{measCQf_p} we may assume that $f$ is convex-quasiconvex.
%On the other hand, as in the proof of Theorem %\ref{maintheoreminftyrelax}, we can localize our functional %$\overline{J}_p$, defining
%$$
%\displaystyle{\overline{F}_p(u,v, A):=\inf \left\{ \liminf_n F(u_n,v_n,A): %u_n\to u \hbox{ in }L^1(A;\mathbb R^n), v_n \rightharpoonup v \hbox{ %in }L^p(A;\mathbb R^m)\right\}}$$
%where
%$$F(u,v,A)=\left\{\begin{array}{l}\displaystyle{\int_A f(x,u(x),v(x),%\nabla u(x))\,dx} \,\text{if}\ (u,v)\in W^{1,1}(A;\mathbb{R}^n)\times %L^p(A;\mathbb{R}^m),\vspace{0.2cm}\\ +\infty \ \text{ in }(L^1(A;%\mathbb R^n)\setminus W^{1,1}(A;\mathbb R^n)) \times L^p(A;%\mathbb R^m).\end{array}\right.$$
%Clearly arguing as in \cite[Theorem 4.3]{AMT} $\overline{F}_p(u,v,%\cdot)$ turns out to be, for every $(u,v) \in BV(\Omega;\mathbb R^n) %\times L^p(\Omega;\mathbb R^m)$ the restriction of a Borel measure, %absolutely continuous with respect to ${\cal L}^N+ |Du|$.
%On the other hand, it is clear that %$\overline{J}_p(u,v)=\overline{F}_p(u,v,\Omega)$.

Observe first that since $f$ fulfills $(H1_p)$ and $(H2_p)$, then it satisfies  $(H1_\infty)$ and $(H2_\infty)$ in the strong form \eqref{H1_1inftyCQf} $-$ \eqref{H2_2inftyCQf}.
Consequently
\begin{equation}\label{ineq}
\overline{J}_p(u,v,\Omega)\leq \overline{J}_\infty(u,v, \Omega)
\end{equation}
 for every $(u, v) \in BV( \Omega;\mathbb R^n) \times L^\infty(\Omega;\mathbb R^m).$

For every positive real number $\lambda$, let $\tau_\lambda: [0,+\infty) \to [0, +\infty)$ be defined as
$$
\tau_\lambda(t)=\left\{
\begin{array}{ll}
t &\hbox{ if } 0 \leq t \leq \lambda,\\
0 &\hbox{ if } t \geq \lambda.
\end{array}
\right.
$$
For every $v \in L^p(\Omega;\mathbb R^m)$, define $v_\lambda:= \tau_\lambda(|v|)v$. Clearly $\int_\Omega |v_\lambda|^p dx \leq \int_\Omega |v|^p dx$ and $v_\lambda \to v$ in $L^p(\Omega;\mathbb R^m)$, as $\lambda \to +\infty$. By the lower semicontinuity of $\overline{J}_p$, \eqref{ineq}, and Theorem \ref{maintheoreminftyrelax}, for every sequence $\{\lambda\}$ such that $\lambda \to +\infty$ we have that
$$
\begin{array}{ll}
\displaystyle{\overline{J}_p(u,v)\leq \liminf_{\lambda \to \infty} \overline{J}_p(u,v_\lambda)=}\displaystyle{\liminf_{\lambda \to +\infty} \int_\Omega f(x,u(x), v_\lambda(x), \nabla u(x))dx}.
\end{array}
$$
Lebesgue's dominated convergence Theorem entails that
$$
\displaystyle{\overline{J}_p(u,v,\Omega)= \int_\Omega f(x,u(x),v(x),\nabla u(x))dx,}
$$
for every $(u,v) \in W^{1,1}(\Omega;\mathbb R^n)\times L^p(\Omega;\mathbb R^m),$
and that concludes the proof.
\end{proof}

\subsection*{Acknowledgments}
The work of A.M. Ribeiro was partially supported by the  Funda\c{c}\~{a}o para a Ci\^{e}ncia e a Tecnologia (Portuguese Foundation for Science and Technology) through PEst-OE/MAT/UI0297/2011 (CMA), UTA-CMU/MAT/0005/2009 and \\ 
PTDC/MAT109973/2009.

\noindent The research of E. Zappale was partially supported by GNAMPA through Project `Problemi variazionali e misure di Young nella meccanica dei materiali complessi' and Universit\'a del Sannio.
The authors thank Irene Fonseca for suggesting the problem.

%%%%% Enter the widest reference label as the first parameter%%%%%

%%% Only the title of article is italicized. No boldface numbers are used.%%%
%% The issue number is only given when the issues are paginated separately.%%%

%\section{References}

\end{document}